\documentclass[11pt]{amsart}

\usepackage{macros}
\usepackage{mathtools}
\usepackage{tikz-cd}

\def\co{\colon\thinspace}

\def\calgd{\mathsf{C}^*}
\def\fgr{\fg_{\text{red}}}

\def\xto{\xrightarrow}

\title{The $\hat{A}$-genus as a projective volume form on the derived loop space}

\author{Ryan Grady}
\address{Department of Mathematical Sciences\\Montana State University\\Bozeman, MT 59717}
\email{ryan.grady1@montana.edu}
\thanks{The author was partially supported by the National Science Foundation under Award DMS-1309118.}


\begin{document}

\begin{abstract}
In the present work, we extend our previous work with Gwilliam by realizing $\hat{A}(X)$ as the projective volume form associated to the BV operator in our quantization of a one-dimensional sigma model. We also discuss the associated integration/expectation map.  We work in the formalism of $L_\infty$ spaces, objects of which are computationally convenient presentations for derived stacks.  Both smooth and complex geometry embed into $L_\infty$ spaces and we specialize our results in both of these cases.

\end{abstract}

\maketitle

\setcounter{tocdepth}{1}
\tableofcontents

\section{Introduction} 

In this paper we complete our description of topological quantum mechanics in the perturbative Batalin--Vilkovisky (BV) formalism. We use the mathematical approach to quantum field theory (QFT) developed by Costello \cite{Cos1} and Costello--Gwilliam \cite{CosGw, CosGw2}. Our theory is formulated as a a one-dimensional sigma model with target the cotangent bundle, $T^\ast X$, of a smooth manifold, $X$.\footnote{In \cite{GGCS} and in the sequel, we call this theory one-dimensional Chern--Simons theory due to its similarity to perturbative Chern--Simons theory in three dimensions, see Section \ref{sect:1d}.  However, due to the nature of the theory, particularly the action functional, one could/should call it a (deformed) BF theory.  See \cite{CMR} for a concrete description of these various theories.} In previous work (with Gwilliam) \cite{GGCS} we defined this theory, proved that it admitted a one-loop quantization and identified the partition function as the $\hat{A}$-genus, $\hat{A} (X)$. With Gwilliam and Williams \cite{GGW}, we proved that the (factorization) algebra of quantum observables is the Rees algebra of the filtered algebra of differential operators on $X$. Since our quantization of one-dimensional Chern--Simons is a {\it cotangent quantization}, one would expect to obtain a projective volume form on the space of fields and an integration theory.\footnote{The relationship between cotangent quantizations and projective volume forms is recalled in Appendix \ref{sect:pvol}.} Here, we prove that this volume form can be identified with $\hat{A} (X)$ and describe the associated integration/expectation map.

The pioneering work of Costello, \cite{CosWG}, in which he recovered the Witten genus via a two-dimensional sigma model, has inspired much of the present work and its precedents. In our work, past and present, we adapt the outline provided in {\it ibid.} to one-dimensional Chern--Simons theory. In the process, fleshing out detail, correcting oversights, and modifying Costello's formalism to work in the smooth (not necessarily holomorphic) setting.

 As discussed below, we work in the formalism of $L_\infty$ spaces \cite{GGLoop}, objects of which are computationally convenient presentations for derived stacks.  Both smooth and complex geometry embed into $L_\infty$ spaces and we specialize our results in both of these cases. Moreover, various flavors of (derived) loop spaces are presented by $L_\infty$ spaces. The loop space $\widehat{\cL_{dR} X}$, see Section \ref{sect:dloop},   plays a central role in the formulation of our field theory.  Moreover, we can state the expectation map as a linear map on functions on $\widehat{\cL_{dR} X}$.  In the case that $X$ is a complex manifold, functions on $\widehat{\cL_{dR} X}$ can be identified with holomorphic differential forms. 
 
In the case of a complex manifold, our main theorem can be stated as follows.

\begin{theorem}
The quantization of one dimensional Chern--Simons theory with target a complex manifold $X$ determines an integrable projective volume form $\overline{dVol}_{S^1}$ on $\widehat{\cL_{dR} X}$ with associated integral
\[
\int_{\widehat{\cL_{dR} X}} \alpha \; \overline{dVol}_{S^1} = k \int_{X} \alpha \cdot \hat{A} (X) ,
\]
for $k$ a nonzero constant and where $\int_X$ indicates the usual integration of holomorphic forms.
\end{theorem}

We have an analogous theorem for an arbitrary $L_\infty$ space $B \fg$.  When $B \fg$ describes a smooth manifold, the na\"{i}ve integral is zero. However, there is an interesting $S^1$-equivariant integration theory, which we describe in Section \ref{sect:smooth}.

Our main theorem is not unexpected from a physical perspective, see \cite{AG, FW, GetzlerIndex, WittenIAS}.  There have also been results of a similar vein in the mathematics community, for instance \cite{Grivaux, Ramadoss, Markarian} or \cite{KSDefQuant} (Chapter 5). The novelty in our result is its derivation via a mathematically rigorous quantum field theory and its interpretation in terms of (smooth) derived stacks.

We note that there is a related one-dimensional theory with target a symplectic manifold $M$. With Si Li and Qin Li \cite{GLL}, we quantized this theory, in the process recovering Fedosov quantization \cite{Fed} and giving an explicit presentation of the Algebraic Index Theorem \cite{NT}. Interestingly, as the general symplectic target is not (globally) a cotangent bundle, there is no one-loop quantization; in fact, there are quantum corrections to all orders.

\subsection{Perturbative BV Theory and Integration}

Let us briefly recall some aspects of integration in quantum field theory (QFT).

In field theory, one studies a space of fields/configurations, $\sE$, and an action functional on the space $S \colon \sE \to \RR$.  Classical field theory/physics focuses on the critical points of $S$, i.e., the solutions to the {\it Euler-Lagrange} equations. (Euclidean) quantum field theory aims to understand the (typically non-existent) measure $e^{-S(\phi)/\hbar} \sD \phi$ on the space of fields $\sE$.
More accurately, one wishes to compute the expected value of functions $\cO$ on the space of fields:
\[
\langle \cO \rangle :=  \int_{\sE} \cO( \phi) e^{-S(\phi)/\hbar} \sD \phi.
\]

In perturbative QFT, we treat the parameter $\hbar$ as formal, not a physical quantity. The idea is that for very small values of $\hbar$, the integrals above receive their dominant contribution from a small neighborhood around the critical points of $S$. In nice cases, the space of critical points is a finite-dimensional manifold (or not too far from one) that we'll call the {\em critical manifold}. Thus, we might hope to approximate the ``true integral" by integrating over a tubular neighborhood of the critical manifold, which we view as an infinite-dimensional normal bundle on the critical manifold.
Ideally, we could use some sort of Fubini Theorem. First, one integrates out the normal bundle (i.e., fix a critical point and integrate over the fiber, which is an infinite-dimensional vector space). Second, integrate the induced measure on the critical manifold. The first stage is an analogue of stationary phase approximation. The second stage is just a usual integral.

 Hence, our task is to do the following.
\begin{enumerate}
\item Define what we mean by a ``perturbative quantum field theory" (i.e., computing around a fixed critical point);
\item Define families of perturbative QFTs over manifolds (i.e., how to integrate out the ``normal direction"); and
\item Explain how to obtain a volume form on the parameterizing ``critical manifold."
\end{enumerate}

In the monograph \cite{Cos1}, Costello develops a mathematical formulation of perturbative QFT which fuses effective field theory and BV theory, thereby addressing task (1). This work was subsequently expanded in \cite{CosGw2} (especially Chapters 5 \& 8); we provide a rapid overview in Appendix \ref{app:BV1}. 

There are many approaches to task (2). In the present work, we will use the language of $L_\infty$ spaces which were introduced in \cite{CosWG}, and subsequently developed by the author and Gwilliam in \cite{GGLoop} and \cite{GGAlgd}. 

Costello also addresses task (3) in \cite{CosWG}, with further details appearing in \cite{CosGw2} (Section 10.4). Necessarily, we make some technical adaptations to this formalism in Sections \ref{sect:nice} and \ref{sect:vol}  in order to apply it to our one-dimensional theory.

\subsection{A Reader's Guide}

For someone familiar with Costello's formalism and $L_\infty$ spaces, it is sufficient to read Sections \ref{sect:1d},  \ref{sect:Avol}, and  \ref{sect:int}, revisiting some technical results in Sections \ref{sect:nice} and \ref{sect:vol} which differ from those in \cite{CosWG}.  

A reader unfamiliar with the work of Costello and Gwilliam may prefer to start with Appendix \ref{app:BV1}. The technical machinery for 
defining our one-dimensional theory is recalled in Section \ref{sect:Loo}, and the remainder of the paper can be read linearly.

Appendix \ref{app:genera} recalls the construction of the $\hat{A}$-genus and its appearance in math and physics.

\subsection{Acknowledgements} The present work grew out of a chapter of the author's PhD dissertation at the University of Notre Dame.  As such, the author is thankful to the university, the math department, and particularly Stephan Stolz and Sam Evens for continued support and discussion. Further, many thanks are due to Kevin Costello and Owen Gwilliam for constant inspiration and collaboration.  The anonymous referees greatly enhanced the presentation and readability of the paper. Lastly, thanks to a great mathematical QFT community including Damien Calaque, Si Li, Qin Li, Brian Williams, and Dan Berwick-Evans.

\section{$L_\infty$ Preliminaries and Notation}\label{sect:Loo}

\subsection{Conventions}

We work throughout in characteristic zero.
We work cohomologically, so the differential in any complex increases degree by one.

For $A$ a cochain complex, $A^\sharp$ denotes the underlying graded vector space. If $A$ is a cochain complex whose degree $k$ space is $A^k$, then $A[1]$ is the cochain complex where $A[1]^k = A^{k+1}$. We use $A^\vee$ to denote the graded dual.

For $V$ a graded $R$-module, its {\em completed symmetric algebra} is the graded $R$-module
\[
\csym_{R} (V) = \prod_{n \geq 0} \Sym^n_{R}(V)
\]
equipped with the filtration $F^k \csym_{R} (V)  = \Sym^{\geq k}_{R}(V)$ and the usual commutative product, which is filtration-preserving. 

\subsection{$L_\infty$ Algebras and Spaces}

Here we provide an overview of our (concrete) setting for derived geometry from \cite{GGLoop}. 

\begin{definition}
Let $A$ be a commutative dg algebra with a nilpotent dg ideal $I$. A {\em curved $\L8$ algebra over $A$} consists of
\begin{enumerate}
\item[(1)] a locally free, $\ZZ$-graded $A^\sharp$-module $V$, and
\item[(2)] a linear map of cohomological degree 1
\[
d: \Sym (V[1]) \to  \Sym (V[1]),
\]
\end{enumerate}
where $\Sym (V[1])$ indicates the graded vector space given by the symmetric algebra over the graded algebra $A^\sharp$ underlying the dg algebra $A$. Further, we require 
\begin{enumerate}
\item[(i)] $d^2 = 0$,
\item[(ii)] $(\Sym (V[1]),d)$ is a cocommutative dg coalgebra over $A$ (i.e., $d$ is a coderivation), and
\item[(iii)] modulo $I$, the coderivation $d$ vanishes on the constants (i.e., on $\Sym^0$).
\end{enumerate}
\end{definition}

We use $C_\ast(V)$ to denote the cocommutative dg coalgebra $(\Sym (V[1]),d)$;  we call it the \emph{Chevalley--Eilenberg homology complex} of $V$, as it extends the usual notion of Lie algebra homology. There is also a natural Chevalley--Eilenberg \emph{cohomology} complex $C^\ast(V)$. It is $(\csym (V^\vee[-1]),d)$, where the notation $\csym (V^\vee[-1])$ indicates the completed symmetric algebra over the graded algebra $A^\sharp$ underlying the dg algebra $A$.  The differential $d$ is the ``dual" differential to that on $C_*(V)$. In particular, it makes $C^*(V)$ into a commutative dg algebra, so $d$ is a derivation.

We now describe a version of ``families of curved $\L8$ algebras parametrized by a smooth manifold."

\begin{definition}
Let $X$ be a smooth manifold. An {\it $\L8$ space} is a pair $(X, \fg)$, where $\fg$ is the sheaf of smooth sections of a $\ZZ$-graded vector bundle $\pi: V \to X$ equipped with the structure of a curved $\L8$ algebra structure over the commutative dg algebra $\Omega^\ast_X$ with nilpotent ideal $\sI = \Omega^{\geq 1}_X$.
\end{definition}

For brevity, we sometimes write $B \fg$ for the $L_\infty$ space $(X, \fg)$ (see the notation below).  By definition, functions on $B \fg$ are given by
\[
\sO (B \fg) \overset{\text{def}}{=} C^\ast (\fg).
\]

An $\L8$ space $B \fg = (X, \fg)$ has an associated ``functor of points'' and hence can be understood as presenting a kind of space $\bB \fg$ in the same way that a commutative algebra presents a scheme. More precisely, to $(X, \fg)$ we associate a simplicial set valued functor
\[
\bB \fg: \dgMan^{op} \to s\!\Sets,
\]
where $\dgMan$ is the site (in fact, the $\infty$-site) of {\it nil dg manifolds}, in which an object $\cM$ is a smooth manifold $M$ equipped with a sheaf $\sO_\cM$ of commutative dg algebras over $\Omega^*_M$ that has a nil dg ideal $\sI_\cM$ such that $\sO_\cM/\sI_\cM \cong \cinf_M$.  For the full definition, including the definition of cover, see \cite{GGLoop}. Moreover, the functor $\bB \fg$ preserves weak equivalences and satisfies descent as embodied in the following result.

\begin{theorem}[Theorem 4.8 \cite{GGLoop}]\label{thm:L8IsDerived}
The functor $\bB \fg$ associated to an $\L8$ space $(X,\fg)$ is a derived stack.
\end{theorem}

Moreover, the association $(X,\fg) \mapsto B \fg$ is part of a functor of categories with weak equivalences from $L_\infty$ spaces to derived stacks \cite{GGAlgd}. Further, this functor also detects weak equivalences.


\begin{example}\label{ex:Loo}
We have several geometric examples of $L_\infty$ spaces.
\begin{enumerate}
\item Consider the $L_\infty$ space $(X, 0)$, for $X$ a smooth manifold. This $L_\infty$ space presents a version of the {\it de Rham stack} $X_{dR}$.  For any dg manifold $(M, \sO_\cM)$, we have $X_{dR} (M, \sO_\cM) = X_{dR} (M, \cinf_M)$, which is the constant simplicial set of smooth maps $M \to X$.
\item Let $X$ be a smooth manifold. There is an $L_\infty$ space $(X, \fg_X)$ such that
\begin{enumerate}
\item $\fg_X \cong \Omega^\sharp_X(T_X[-1])$ as  $\Omega^\sharp_X$ modules, and
\item $C^*(\fg_X) \cong dR(\sJ) $ as commutative $\Omega_X$ algebras, where $\sJ$ is the infinite jet bundle.
\end{enumerate}
As we explain in \cite{GGLoop}, $\bB \fg_X$ is a natural {\it derived enhancement} of the smooth manifold $X$. 
\item Let $Y$ be a complex manifold, then there exists an $\L8$ space $(Y, \fg_{Y_{\overline{\partial}}})$, such that
\begin{enumerate}
\item As an $\Omega_Y^\sharp$-module, $\fg_{Y_{\overline{\partial}}}$ is isomorphic to $\Omega^\sharp_Y(T^{1,0}_Y [-1])$;
\item The derived stack $\bB \fg_{Y_{\overline{\partial}}}$ represents the moduli problem of holomorphic maps into $Y$, i.e., for any complex manifold $Z$ (viewed as a nilpotent dg manifold), $\bB \fg_{Y_{\overline{\partial}}}(Z)$ is the discrete simplicial set of holomorphic maps from $Z$ to $Y$.
\end{enumerate}
\item Generalizing the previous constructions, in \cite{GGAlgd} we associate an $L_\infty$ space to any Lie algebroid.  That is, let $\rho \co L \to T_X$ be a Lie algebroid over a smooth manifold $X$, then there exists an $\L8$ space $(X, \fg_L)$
such that
 \begin{enumerate}
 \item $\fg_L \cong \Omega^\sharp_X (T_X[-1] \oplus L)$ as $\Omega^\sharp_X$ modules, and
 \item $C^* ( \fg_L) \cong dR(J(\calgd(L)))$ as commutative $\Omega^\ast_X$ algebras.
 \end{enumerate}
\end{enumerate}
\end{example}

\subsection{Vector Bundles on $L_\infty$ Spaces}

Many natural constructions in geometry work in the setting of $L_\infty$ spaces, as we briefly recall.

\begin{definition}\label{defn:vb}
Let $(X, \fg)$ be an $\L8$ space.  A {\it vector bundle} on $(X,\fg)$ is a $\ZZ$-graded vector bundle $\pi:V \to X$ where the sheaf of smooth sections $\cV$ over $X$ is equipped with the structure of an $\Omega^\sharp_X$-module and where the direct sum of sheaves $\fg \oplus \cV$ is equipped with the structure of a curved $\L8$ algebra over $\Omega^*_X$, which we denote $\fg \ltimes \cV$, such that
\begin{enumerate}
\item[(1)] the maps of sheaves given by inclusion $\fg \hookrightarrow  \fg \ltimes \cV$ and by the projection $\fg \ltimes \cV \to \fg$ are maps of $\L8$ algebras, and
\item[(2)] the Taylor coefficients $\ell_n$ of the $\L8$ structure vanish on tensors containing two or more sections of $\cV$.
\end{enumerate}
The {\it sheaf of sections of $\cV$ over $(X,\fg)$} denotes $C^* (\fg, \cV[1])$, the sheaf on $X$ of dg $C^*(\fg)$-modules given by the Chevalley--Eilenberg complex of $\cV$ as a $\fg$-module. The {\it total space} for the vector bundle $\cV$ over $(X,\fg)$ is the $\L8$ space $(X, \fg \ltimes \cV)$.
\end{definition}

For example, the tangent bundle to $(X, \fg)$ is given by $\fg[1]$ equipped with the adjoint action of $\fg$. Dually, the cotangent bundle is given by $\fg^\vee [-1]$ equipped with the coadjoint action.  It follows that the $k$-forms on $(X, \fg)$ are given by 
\[
\Omega^k_{(X,\fg)} = C^* (\fg , (\Lambda^k \fg) [-k]),
\]
as discussed in \cite{GGLoop}.

Recall that a symplectic form is a 2-form that is nondegenerate and closed. This definition works perfectly well in the derived setting, so long as one recognizes that being closed --- i.e., being annihilated by the differential of the de Rham complex --- is data and not a property.

Let $\Omega^{2,cl}_{(X,\fg)}$, the complex of {\em closed 2-forms} on the $\L8$ space, be the totalization of the double complex
\[
\Omega^2_{(X,\fg)} \xto{d_{dR}} \Omega^3_{(X,\fg)} \xto{d_{dR}} \Omega^4_{(X,\fg)} \xto{d_{dR}} \cdots.
\]
A {\em closed 2-form} is a cocycle in this complex. Every element $\omega$ of $\Omega^{2,cl}_{(X,\fg)}$ has an underlying 2-form $i(\omega)$ by taking its image under the truncation map $i \co \Omega^{2,cl}_{(X,\fg)} \to \Omega^2_{(X,\fg)}$.

\begin{definition}
An {\em $n$-shifted symplectic form} on an $\L8$ space $(X,\fg)$ is a closed 2-form $\omega$ of cohomological degree $n$ such that the induced map $i(\omega) \co T_{(X,\fg)} \to T^*_{(X,\fg)}[-n]$ is a quasi-isomorphism.
\end{definition}

\subsection{The Derived Loop Space} \label{sect:dloop}

In derived geometry, there are several flavors of circle, hence there are several loop spaces. We present a few of these loop spaces as $L_\infty$ spaces. For further discussion, see \cite{GGLoop}, \cite{BZN}, or \cite{TV11}.

Recall from Example \ref{ex:Loo} above, that a smooth manifold $X$ can be enhanced to a derived stack $\bB \fg_X$. We therefore can define a derived enhancement of the smooth loop space as follows
\[
\begin{array}{cccc}
\cL_{sm}X: & \dgMan^{op} & \to & s\!\Sets \\[1ex]
& \cM & \mapsto &  \bB\fg_X(S^1 \times \cM)
\end{array}.
\]
This space is a derived stack, but it doesn't have a presentation in terms of $L_\infty$ spaces and as such doesn't manifestly define a perturbative gauge theory in our formalism (see Section \ref{sect:BVloop}).

Next, consider the $L_\infty$ space $(X, \RR[\epsilon] \otimes \fg_X)$ for $\epsilon$ a square zero parameter of degree 1.  We call this space the {\it Betti loop space} of $X$ and denote it $\cL_{\cB} X$. We have an isomorphism of $L_\infty$ spaces $\cL_{\cB} X \cong T[-1] \bB \fg_X$. The Betti loop space should be thought of as the mapping object obtained by replacing $S^1$ by its cohomology ring.

The final version of the circle we consider is $S^1_{dR} = (S^1 , \Omega^\ast (S^1))$.
This flavor of $S^1$ gives us the {\it de Rham loop space}  of $X$:
\[
\begin{array}{cccc}
\cL_{dR}X: & \dgMan^{op} & \to & s\!\Sets \\[1ex]
& \cM & \mapsto &  \bB\fg_X(S^1_{dR} \times \cM)
\end{array}.
\]
The de Rham loop space again doesn't have a presentation in terms of an $L_\infty$ space, however a certain substack does.  Consider the substack $\widehat{\cL_{dR} X}$, presented by the $L_\infty$ space $(X, \Omega^\ast (S^1) \otimes \fg_X)$.

\begin{prop}[\cite{GGLoop} Lemma 6.6/6.7]
Let $X$ be a smooth manifold and consider the $L_\infty$ space $\widehat{\cL_{dR} X} = (X, \Omega^\ast(S^1) \otimes \fg_X)$.
\begin{enumerate}
\item The derived stack presented by $\widehat{\cL_{dR} X}$ is the substack of $\cL_{dR}X$ where for any dg manifold $(M, \sO_M)$ the underlying map of smooth manifolds $S^1 \times M \to X$ is constant along $S^1$.
\item Any volume form $\omega$ on $S^1$ determines a weak equivalence of derived stacks
\[
\underline{\omega} :  \cL_{\cB} X \Rightarrow \widehat{\cL_{dR} X}.
\]
\end{enumerate}
\end{prop}

Finally, if $X$ is a symplectic manifold, then $\widehat{\cL_{dR}X}$ is a -1-symplectic derived stack. Explicitly,  fix a symplectic form $\omega \in \Omega^2(X)$ and a 1-form $\nu \in \Omega^1(S^1)$ that is not exact, then consider the pairing
\[
\begin{array}{cccc}
\Omega_{\omega,\nu}:&[\fg_X \ot \Omega^*(S^1)]^{\ot 2} &\to &\Omega^*(X)\\
&(Z \ot \alpha) \ot (Z' \ot \alpha') &\mapsto &{\displaystyle \int_{\theta \in S^1} }J(\omega)(Z \ot \alpha(\theta), Z' \ot \alpha'(\theta)) \wedge \nu(\theta)
\end{array}.
\]

\begin{prop}[\cite{GGLoop} Lemma 6.8]
The 2-form $\Omega_{\omega, \nu}$ is a -1-symplectic form on $\widehat{\cL_{dR}X}$.
\end{prop}

\subsection{Nice $L_\infty$ Spaces}\label{sect:nice}

Following \cite{CosWG}, we describe some tameness properties of $\L8$ spaces.  These properties will be used in our discussion of integrable volume forms in Section
\ref{sect:vol}.

\begin{definition}
Given an $\L8$ space $(X,\fg)$, the reduced structure sheaf $\fg_{\mathrm{red}}$ is defined by
\[
\fgr = \fg/\Omega^{>0}_X .
\]
\end{definition}

\begin{prop}
Given an $\L8$ space $(X, \fg)$, the reduced structure sheaf $\fgr$ has no curving i.e. $l_1^2 =0$.
\end{prop}

\begin{proof}
From the $\L8$ relations we know that $l_1^2 = l_0$.  Now $l_0 : \CC \to V$ i.e. $l_0$ is just an element of $V$ which is dual to the map $d_0 : V^\vee [-1] \to \CC$.  The condition that  reduced modulo the nilpotent ideal $I$ the derivation $d$ preserves the ideal generated by $V$ implies that $l_0 \in V \otimes_A I$.  Therefore reduced modulo $I$, $l_0 =0$.
\end{proof}

\begin{definition}
\mbox{}
\begin{enumerate}
\item An $\L8$ space $(X, \fg)$ is {\it locally trivial} if the $C^\infty_X$-linear sheaf of $\L8$-algebras $\fgr$ is locally quasi-isomorphic to the sheaf of sections of a graded vector bundle $V$, with trivial differential and $\L8$ structure.
\item An $\L8$ space $(X, \fg)$ is {\it quasi-smooth} if the cohomology sheaves of $\fgr$ are concentrated in degrees 1 and 2.
\item An $\L8$ space $(X, \fg)$ is {\it nice} if it is both locally trivial and quasi-smooth.
\end{enumerate}
\end{definition}

Considered up to equivalence, the local structure of $\fgr$ for a nice $\L8$ space $(X, \fg)$ is easy to describe.  Indeed by assumption $(X, \fg)$ is nice and hence locally trivial, so we can assume that locally $\fgr$ has trivial differential and $\L8$ structure, i.e.,  locally $\fgr$ is a free $C^\infty_X$-module.  Let $d_i$ denote the rank of $H^i (\fgr )$. As $(X, \fg)$ is additionally quasi-smooth $d_i = 0$ for $i \neq 1,2$.  Let $V$ be the graded vector space given by
\[
V = \CC^{d_1} \oplus \CC^{d_2} [1] .
\]
Locally we have an isomorphism 
\[
\fgr \cong V^\vee [-1] \otimes_\CC C^\infty_X .
\]
The Chevalley--Eilenberg complex also has a description in terms of $V$:
\[
C^\ast (\fgr) \cong \csym_{C^\infty_X} ( (V^\vee [-1] \otimes_\CC C^\infty_X)^\vee [-1]) \cong C^\infty_X \otimes_\CC \csym_\CC (V) .
\]


\begin{lemma}
If $(X, \fg)$ is a nice $\L8$ space, then $T^\ast [-1] (X, \fg) = (X , \fg \oplus \fg^\vee [-3] )$ is also nice.
\end{lemma}

\begin{proof}
Note that as $C^\infty_X$-modules we have $(\fg \oplus \fg^\vee [-3])_\mathrm{red} \cong (\fgr \oplus \fgr^\vee [-3])$.  As $(X,\fg)$ is locally trivial, we a have a local isomorphism $\fgr \cong V^\vee [-1] \otimes_\CC C^\infty_X$ for a finite dimensional graded vector space $V$ and consequently we have a local isomorphism 
\[
(\fg \oplus \fg^\vee [-3])_\mathrm{red} \cong (V^\vee[-1] \oplus V[-2]) \otimes_\CC C^\infty_X 
\]
expressing the local triviality of $T^\ast[-1] (X, \fg)$.

That $(X,\fg)$ is nice and in particular quasi-smooth implies that $V$ is concentrated in degree 0 and $-1$.  Since $(\fg \oplus \fg^\vee [-3])_\mathrm{red}$ is quasi-isomophic locally to an $\L8$ algebra with no differential (or higher brackets) quasi-smoothness of $T^\ast [-1] (X, \fg)$ follows from observing that $V^\vee [-1] \oplus V[-2]$ (and hence the cohomology sheaves) is again concentrated in degrees 1 and 2.
\end{proof}

The same technique proves the following.

\begin{prop}\label{smoothnice}
\mbox{}
\begin{itemize}
\item Let $(X, \fg_X)$ denote the $\L8$ space encoding the smooth geometry of $X$, then $T[-1] (X, \fg_X)$ is nice.
\item Let $(Y, \fg_{Y_{\overline{\partial}}})$ be the $\L8$ space encoding the complex structure of a complex manifold $Y$, then $T[-1] (Y, \fg_{Y_{\overline{\partial}}})$ is nice.
\end{itemize}
\end{prop}

In general if $(X, \fg)$ is nice, $T[-1] (X, \fg)$ is not necessarily nice as we may have non-trivial degree 3 cohomology.

\section{One-Dimensional Chern--Simons Theory and its Quantization}\label{sect:1d}

Here we briefly describe a family of field theories which we call one-dimensional Chern--Simons theories.  We are working in Costello's paradigm of effective BV theory as described in \cite{Cos1} (and Appendix \ref{app:BV2}), see also \cite{GGCS} for further details on Chern--Simons type theories.

Although Chern--Simons theory typically refers to a gauge theory on a 3-manifold, the perturbative theory has analogues over a manifold of any dimension. The only modification is to use dg Lie algebras, or $\L8$ algebras, with an invariant pairing of the appropriate degree.

\subsection{The Classical Theory}\label{sect:1dcs}

Let $\fg$ denote a curved $\L8$ algebra over a commutative dga $A$, and let $\{\ell_n\}_{n \ge 0}$ denote the $n$-ary brackets of $\fg$. Note that the sum $\fg \oplus \fg^\vee[-2]$ is equipped with an $\L8$ structure using the coadjoint action.
Moreover, $\fg \oplus \fg^\vee[-2]$ also has a natural pairing
\[
\langle \xi + \lambda, \zeta + \mu \rangle = \lambda(\zeta) - \mu(\xi),
\]
which is invariant by construction.

Our Chern--Simons theory then has a space of fields 
\[
\Omega_{S^1} \ot \left( \fg[1] \oplus \fg^\vee[-1] \right).
\]
The (classical) action functional is given by
\[
S(\phi) = \frac{1}{2}\langle \phi, d \phi \rangle + \sum_{n = 0}^\infty \frac{1}{(n+1)!} \langle \phi, \ell_n(\phi^{\ot n}) \rangle.
\]
The Euler-Lagrange equation of $S$ is the Maurer--Cartan equation for the trivial $\fg \oplus \fg^\vee[-2]$-bundle on $S^1$.
We view the action $S$ as a sum of a free action functional $\frac{1}{2}\langle \phi, (d + l_1) \phi \rangle $ and an interaction term $I_{CS} = \sum_{n \neq 1} \frac{1}{(n+1)!} \langle \phi, \ell_n(\phi^{\ot n}) \rangle.$ Note that this theory is a cotangent theory.

\subsection{Quantization of One-Dimensional Chern--Simons}\label{sect:1dcsquant}

We recall the main theorems from \cite{GGCS} which describe the global classical and quantum observables for one-dimensional Chern--Simons theory determined by an $\L8$ algebra $\fg$. (It is non-trivial to see that a quantization exists after which we can begin to identify the complex of global observables.) To begin, we note that the classical observables on $S^1$ are given by the commutative dg algebra
\[
\left ( \csym \left ( \left ( \Omega^\ast_{S^1} \otimes \left ( \fg [1] \oplus \fg^\vee [-1] \right) \right)^\vee \right ) , d + \{ I_{CS} , - \} \right ),
\]
where $I_{CS}$ is the interacting part of the action functional as described in preceding section.

\begin{theorem}\label{thm:QCS}
The one-dimensional Chern--Simons theory determined by $\fg$ admits a BV quantization.  Moreover, this quantization is a cotangent quantization (in the sense of \cite{CosGw2}).
\end{theorem}

The proof of the preceding is by explicit obstruction calculations contained in \cite{GGCS}.

\begin{theorem}\label{LieThm}
The global quantum observables of the Chern--Simons theory  determined by $\fg$ on $S^1$ are quasi-isomorphic to the following cochain complex:
\[
(\Omega^{-*}_{T^*B\fg}[[\hbar]], \hbar L_\Pi + \hbar\{\log (\hat{A}(B\fg)), -\}),
\]
where $L_\Pi$ denotes the Lie derivative with respect to the canonical Poisson bivector $\Pi$ on $T^* B\fg$. 
\end{theorem}

\begin{cor}
Let $X$ be a complex manifold and ${\displaystyle \fg_{X_{\overline{\partial}}}}$ the $\L8$ algebra which encodes the complex geometry of $X$.  The global quantum observables on $S^1$ for the one-dimensional Chern--Simons theory determined by ${\displaystyle \fg_{X_{\overline{\partial}}}}$ are quasi-isomorphic to
\[
\left ( \Omega^{-\ast}_{hol} (T^\ast X) [[\hbar]], \hbar L_\Pi + \hbar \left \{ \log \left (e^{-c_1 (X)/2} \mathrm{Td} (X) \right ), - \right \} \right ),
\]
where $\Omega^k_{hol}$ denotes the holomorphic $k$-forms.
\end{cor}

\subsection{Relationship to the Derived Loop Space}\label{sect:BVloop}

Recall from Section \ref{sect:dloop}, that $\Omega^\ast(S^1) \otimes \fg$ corresponds to functions on the derived loop space $\widehat{\cL_{dR} B \fg}.$ Equivalently, this $L_\infty$ algebra parametrizes the formal moduli of (nearly) constant maps $\{S^1 \to B \fg\}$. The one-dimensional Chern--Simons theory described above is just the cotangent theory associated to this moduli problem.  

Let $X$ be a smooth manifold and $\fg_X$ the associated sheaf of $L_\infty$ algebras.  In this case, $\widehat{\cL_{dR} X}$ describes a family of formal moduli problems and hence, after passing to the associated cotangent theory, a family of field theories over $X$. This family of theories corresponds to the moduli of maps $\{S^1 \to X \subset T^\ast X\}$.  

Finally, note that $\fg[1] \oplus \fg^\vee [-1]$ has a degree zero symplectic structure.  As such, it leads to a one-dimensional theory via the AKSZ construction \cite{AKSZ}.  In \cite{GGLoop}, we show this construction again leads to our one-dimensional Chern--Simons theory.

\section{Projective Volume Forms}\label{sect:vol}

In this section we discuss the notion of projective volume forms on $\L8$ spaces, setting the stage for interpreting $\hat{A} (B \fg)$ as such a volume form on $T[-1] (X, \fg)$.

Let $B \fg = (X, \fg)$ be an $\L8$ space.  Motivated by complex geometry, Costello \cite{CosSUSY} makes the following definition.\footnote{See also Section 10.4 \cite{CosGw2} for an explicit relationship between cotangent quantizations and projective volume forms.}

\begin{definition}\label{defn:PVF}
A \emph{projective volume form} on $(X, \fg)$ (equivalently on $B\fg$) is a right $D(B \fg)$-module structure on $\sO(B \fg)$.
\end{definition}

Here $D(B \fg)$ is the associative algebra of differential operators on $B \fg$.  Recall that the dg Lie algebra of vector fields on $B \fg$ is given by 
\[
\mathrm{Vect}(B \fg) = C^\ast (\fg , \fg[1] ) = \mathrm{Der} (\sO (B \fg)).
\]
Then we define $D(B \fg)$ to be the free associative algebra over $\sO (B \fg)$ generated by $\chi \in \mathrm{Vect} (B \fg)$ subject to the relations:
\[
\begin{array}{rll}
\chi \cdot f - f \cdot \chi &=& (\chi f)\\
f \cdot \chi &=& f \chi
\end{array}
\]
where $\cdot$ denotes the associative product in $D(B \fg)$ and juxtaposition indicates the (left) action of $\mathrm{Vect} (B \fg)$ on $\sO (B \fg)$ by derivations or the (left) module structure of $\mathrm{Vect} (B \fg)$ over $\sO (B \fg)$.

\begin{prop}[Proposition 11.7.1 of \cite{CosSUSY}]\label{prop:projvol}
There is a bijection between the set of right $D(B \fg)$-structures on $\sO(B \fg)$ and that of $\CC^\times$-equivariant quantizations of the $P_0$ algebra $\sO (T^\ast [-1] B \fg)$.
\end{prop}

\subsection{Integrable Volume Forms}

As we saw above, we can think of the BV laplacian as a divergence operator.  That is, if $\omega$ is a volume form on $T^\ast [-1] X$, then the operator $\Delta_\omega$ is given by $\mathrm{Div}_\omega$.  Turning this correspondence around we can ask when a projective volume form $\omega$ on $(X, \fg)$ is {\it integrable} and hence leads to an appropriate integral 
\[
\int : H^0 (\sO (X, \fg)) \to \CC .
\]

Let $\omega$ be a projective volume form on $(X, \fg)$ corresponding the quantization of $T^\ast [-1] (X, \fg)$ with BV laplacian $\Delta_\omega$ i.e. the bracket of the BD algebra is given by
\[
\{a,b\} = \Delta_\omega (\alpha \beta) - (\Delta_\omega \alpha) \beta - (-1)^{\lvert \alpha \rvert} \alpha (\Delta_\omega \beta) .
\]

\begin{definition}
The {\it divergence complex} associated to $\omega$ is defined by
\[\mathrm{Div}^\ast (\omega) = ( \sO (T^\ast [-1] (X, \fg))((\hbar)), d + \hbar \Delta_\omega).
\]
We let $\cH^i (\mathrm{Div}^\ast (\omega))$ denote the i'th cohomology sheaf of the divergence complex.
\end{definition}

\begin{lemma}
For each $i$, $\cH^i (\mathrm{Div}^\ast (\omega))$ is a sheaf of $\CC ((\hbar))$-modules and carries a $\CC^\times$ action lifting the action of $\CC^\times$ on $\CC ((\hbar))$ where $\hbar$ has weight $-1$.
\end{lemma}

The following lemma shows that $\mathrm{Div}^\ast (\omega)$ is quasi-isomorphic to a local system of $\CC((\hbar))$ lines.  Later we will take $\CC^\times$ invariants to obtain a system of $\CC$ lines which in good cases, i.e., when $\omega$ is {\it integrable}, we can identify with the orientation local system on $X$.

\begin{lemma}[7.8.1 of \cite{CosWG}]
Let $(X, \fg)$ be a nice $\L8$ space, and let $d_i$ denote the rank of $H^i (\fgr)$.  Then, for any projective volume form $\omega$ on $(X, \fg)$, the cohomology sheaves $\cH^i (\mathrm{Div}^\ast (\omega))$ are zero except for $i = -d_1 -d_2.$  Further, $\cH^{-d_1-d_2} (\mathrm{Div}^\ast (\omega))$ is a locally constant sheaf of $\CC((\hbar))$ vector spaces of rank one.
\end{lemma}

We noted above, that in general $T[-1] (X, \fg)$ is not nice, even if $(X, \fg)$ is nice. As a preliminary we need the following lemma which is a slightly more general version of the preceding lemma.

\begin{lemma}
Let $(X, \fg)$ be a nice $\L8$ space. Then for any projective volume form $\omega$ on $T[-1] (X,\fg)$, the cohomology sheaves $\cH^i (\mathrm{Div}^\ast (\omega))$ are zero unless $i = -2d_1$, where $d_1 = \dim H^1 (\fgr)$.  Further, $\cH^{-2d_1} (\mathrm{Div}^\ast (\omega))$ is a locally constant rank one sheaf of $\CC ((\hbar))$ vector spaces.
\end{lemma}

\begin{proof}
We compute locally on $X$ the cohomology $\sO(T^\ast[-1] T[-1] (X, \fg) )[[\hbar]]$ with differential $d+\hbar \Delta_\omega$.  We filter $\sO(T^\ast[-1] T[-1](X, \fg))$ by the image of multiplication by $\Omega^i_X$.  We can compute the cohomology by the associated spectral sequence. The first page is given by 
\[
\oplus \Omega^i_X [-i] \otimes_{\cinf_X} C^\ast (\fgr [\epsilon] \oplus \fgr^\vee [\epsilon] [-2]).
\]

As $(X, \fg)$ is nice, we can assume (since we are computing locally) that $\fgr$ has trivial differential and $\L8$ structure and is free as a $\cinf_X$ module.  Therefore, 
\[
C^\ast (\fgr [\epsilon] \oplus \fgr^\vee [\epsilon] [-2] ) \cong \cinf_X \otimes \csym(V\oplus V^\vee [1])
\]
where $V = V_0 \oplus V_{-1} \oplus V_{-2}$ and the subscripts indicate in which degree the vector space lives.  Let's fix a basis $x_i, \alpha_j , h_k$ of $V_0 \oplus V_{-1} \oplus V_{-2}$ and a dual basis $x^i , \alpha^j , h^k$ of $V^\vee[1]$, so the $x^i$ have degree -1 and so on. If we let $d_i = \mathrm{rank} \; \cH^i (\fgr)$, then $\dim V_0 = d_1$, $\dim V_{-1} = d_1 +d_2,$  and $\dim V_{-2} = d_2$. Then
\[
\csym(V \oplus V^\vee [1]) = \CC [[x_i , x^i , \alpha_j , \alpha^j , h_k , h^k ]]
\]
and by assumption we have a $P_0$ structure given by
\[
\{x_i , x^i \}=1=\{\alpha_j , \alpha^j\}=1=\{h_k , h^k\}
\]
with all other brackets vanishing.

Now we have a $\CC^\times$ equivariant quantization corresponding to $\omega$.  Such a quantization is defined by the operator $\Delta_\omega$, where
\[
\Delta_\omega = \Delta_0 + \{S, -\}
\]
for $S \in \CC[[x_i]]$ and 
\[
\Delta_0 = \sum_i \frac{\partial}{\partial x_i} \frac{\partial}{\partial x^i} + \sum_j \frac{\partial}{\partial \alpha_j} \frac{\partial}{\partial \alpha^j} + \sum_k \frac{\partial}{\partial h_k} \frac{\partial}{\partial h^k} .
\]

We now compute the cohomology of $\CC[[x_i , x^i , \alpha_j , \alpha^j , h_k , h^k ,\hbar]]$ with respect to the differential $\hbar \Delta_0 + \hbar \{S, -\}$.  We will do this via a spectral sequence associated to the filtration $F^l \CC [[x_i , x^i , \alpha_j , \alpha^j , h_k , h^k , \hbar]]$ where the $lth$ filtered piece consists of elements of weight greater than or equal to $l$.  Here we give all the generators weight 1 and $\hbar$ weight 2.  Note that the differential $\hbar \Delta$ preserves this filtration.

The operator $\hbar \Delta_0$ preserves weight, while $\hbar \{S, -\}$ strictly increases weight.  Hence the first page of the spectral sequence is given by
\[
E^{p,q}_1 = H^q (F^p \CC[[x_i, x^i , \alpha_j , \alpha^j , h_k , h^k , \hbar]] , \hbar \Delta_0).
\]

In order to apply the formal Poincar\'e lemma we translate into the language of forms and polyvector fields.  That is let $\partial_{x_i}$ and  $\partial_{\alpha^j}$ be of degree -1 and $\partial_{h_k}$ be of degree 1, then we have an isomorphism
\[
\CC[[x_i, x^i , \alpha_j , \alpha^j , h_k , h^k , \hbar]] \cong \CC[[x_i , \partial_{x_i} , \partial_{\alpha^j} , \alpha^j , h_k , \partial_{h_k} , \hbar]]
\]
which sends
\[
\begin{array}{c}
x^i \mapsto \partial_{x_i} \\[1ex]
\alpha_j \mapsto \partial_{\alpha^j} \\[1ex]
h^k \mapsto \partial_{h_k} .
\end{array}
\]
Define the translation invariant volume form
\[
dVol = dx_1 \wedge \dotsb \wedge dx_{d_1} \wedge d \alpha^1 \wedge \dotsb \wedge d \alpha^{d_1 + d_2} \wedge d h_1 \wedge \dotsb \wedge d h_{d_2} .
\]
Then under the above isomorphism $\Delta_0$ corresponds to the divergence with respect to $dVol$. Contraction with $dVol$ turns polyvector fields into forms so we have an isomorphism
\[
\CC[[x_i , x^i , \alpha_j , \alpha^j , h_k , h^k , \hbar]] \cong \CC[[x_i , \alpha^j , h_k , dx_i , d\alpha^j , dh_k , \hbar]][2d_1]
\]
which is an isomorphism of complexes after equipping the right hand side with the de Rham differential.

After inverting $\hbar$ the Poincar\'e lemma shows that  the cohomology is one dimensional and of degree $-2d_1$.
\end{proof}

\begin{remark}
The above lemma has an analog for any $\L8$ space such that $\fgr$ has bounded cohomology.
\end{remark}

In light of Proposition \ref{smoothnice} we need to check the compatibility of the two preceding lemmas.  For $(X, \fg_X)$, $H^i (\fgr)$ is zero unless $i=1$ and in this case has rank $d_1 = \dim X$.  Then 
\[
T[-1] (X, \fg_X) \cong (X, \fg_X \oplus \fg_X [-1]),
\]
so the cohomology of the reduced $\L8$ algebra is concentrated in degrees 1 and 2 with each group being of rank $\dim X$ i.e. $d_1 = d_2 = \dim X$.  Hence, we see that the lemmas concur in this case.

\subsection{Volume Forms on $T[-1] (X, \fg)$}

As noted above, $\cH^{-2d_1} (\mathrm{Div}^\ast (\omega))$ forms a local system of $\CC ((\hbar))$ lines.  By taking $\CC^\times$ invariants we obtain a system of $\CC$ lines (since $\hbar$ has weight $-1$).  Let
\[
\cD (\omega) = \cH^{-2d_1} (\mathrm{Div}^\ast (\omega))^{\CC^\times}
\]
denote this local system of $\CC$ lines.  If we weight $\hbar$ appropriately we have a quasi-isomorphism of sheaves of $\CC((\hbar))$-modules
\[
\cD (\omega) ((\hbar)) [2d_1] \cong \mathrm{Div}^\ast (\omega).
\]

In order to define integration, see below, we need to be able to identify $\cD (\omega)$ with the orientation local system of $X$.

\begin{definition}
A projective volume form $\omega$ on $T[-1] (X, \fg)$ for $(X, \fg)$ nice is {\it integrable} if the local system $\cD (\omega)$ on the manifold $X$ is isomorphic to the orientation local system on $X$.
\end{definition}

\subsubsection{The canonical volume form}

We expand on the sketch in \cite{CosWG} to show that $T[-1] (X, \fg)$ carries a canonical volume form.

\begin{lemma}
Let $(X, \fg)$ be an $\L8$ space, then we have a natural equivalence $T^\ast[-1] T[-1] (X, \fg) \cong T[-1] T^\ast (X, \fg)$.
\end{lemma}

\begin{proof}
\[\begin{array}{lll}
T^\ast[-1] T[-1] (X, \fg) &=&\left (X, \fg[\epsilon] \oplus \left ( \fg [\epsilon ] \right)^\vee [-3] \right ) \\[2ex]
&=& \left (X, \fg [\epsilon] \oplus \fg^\vee [\epsilon] [-2]  \right ) \\[2ex]
&=& \left (X, \left (\fg \oplus \fg^\vee [-2] \right ) [\epsilon] \right ) \\[2ex]
&=& T[-1] T^\ast (X, \fg) .
\end{array}
\]
There is some subtlety in the identifications above, we need to define the relevant $\L8$ structures.  The $\L8$ structure on $\fg[\epsilon]$ is the one obtained by the product of the commutative algebra $\CC[\epsilon]$ and $\L8$ algebra $\fg$. Recall that we have the coadjoint action of $\fg$ on $\fg^\vee$, we extend this $\epsilon$-linearly to define an $\L8$ structure on $\fg[\epsilon] \oplus \fg^\vee [\epsilon][-2]$.  The degree -3 pairing on $(\fg \oplus \fg^\vee [-2]) [\epsilon]$ is the composition of the $\epsilon$-linear $\CC[\epsilon]$ valued degree -2 pairing with the degree -1 map $\CC [\epsilon] \to \CC$ sending $\epsilon$ to 1.  With these structures just defined, the identifications above are  natural equivalences of $\L8$ spaces.
\end{proof}

\begin{prop}
$T[-1] (X, \fg)$ carries a canonical volume form $\mathrm{dVol}_0$.  That is there exists a square zero operator $\Delta_0$ on $C^\ast (\fg [\epsilon] \oplus \fg^\vee [\epsilon] [-2])$ such that $[d, \Delta_0] =0$, $\Delta_0$ has weight one with respect to the $\CC^\times$ action, and the failure of $\Delta_0$ to be a derivation is the Poisson bracket.
\end{prop}

\begin{proof}
Let $K \in \fg \otimes \fg^\vee$ be the inverse to the canonical pairing.  Define the skew-symmetric tensor
\[
\widetilde{K} \overset{def}{=} (\epsilon \otimes 1 + 1 \otimes \epsilon)K \in \left (\fg [\epsilon] \oplus \fg^\vee [\epsilon][-2] \right )^{\otimes 2} .
\]
Define $\Delta_0$ to be contraction with $\widetilde{K}$.
\end{proof}

Recall that 
\[
C^\ast (\fg[\epsilon] \oplus \fg^\vee [\epsilon][-2] ) = \csym \left ( \fg [\epsilon] [1] \oplus \fg^\vee \epsilon [-1] \right )^\vee
\]
so $\Delta_0$ is the unique differential operator which vanishes on constant and linear tensors and is contraction with $\widetilde{K}$ on quadratic tensors.  $\Delta_0$ can also be defined as
\[
\Delta_0 = [ d_{dR} , i_\Pi] , 
\]
where $\Pi$ is the Poisson tensor on $T^\ast (X, \fg)$ and after identifying $\sO (T[-1]T^\ast (X, \fg))$ with forms on $T^\ast (X, \fg)$, $d_{dR}$ denotes the de Rham differential (which lowers degree by -1).

\section{$\hat{A} (B \fg)$ as a volume form}\label{sect:Avol}

Let $B \fg = (X, \fg)$ be an  $\L8$ space, then we can consider the one dimensional Chern--Simons theory with space of fields 
\[
\widehat{\cL_{dR} T^\ast B \fg} = (X, \Omega^\ast_{S^1} \otimes (\fg \oplus \fg^\vee [-2])) \cong \cL_{\cB} T^\ast B \fg \cong T[-1] T^\ast B \fg \cong T^\ast[-1] T[-1] B \fg .
\]
We saw in Section \ref{sect:1dcsquant} that we can quantize this theory. Therefore, we obtain a projective volume form on 
\[
\widehat{\cL_{dR} B \fg}\cong T[-1] (X, \fg).
\]

\begin{definition}
Let $(X, \fg)$ be an $\L8$ space, then we define $dVol_{S^1}$ to be the projective volume form determined by the quantization of one dimensional Chern--Simons theory with space of fields $T[-1] T^\ast (X, \fg)$.  We let $\Delta_{S^1}$ denote the corresponding BV laplacian.
\end{definition}

It follows from Theorem \ref{LieThm} that we have
\[ \Delta_{S^1} = \Delta_0 + \{ \log (\hat{A} (B \fg)), - \} .\]

\begin{lemma}\label{lem:gaugetransform}
Let $\Delta_\omega = \Delta_0 +\{S_\omega, - \}$ with $S_\omega \in \sO(T[-1] (X, \fg))$ satisfying the QME. Then
the map given by multiplication by $e^{S_\omega /\hbar}$ 
\[
e^{S_\omega/\hbar} : \mathrm{Div}^\ast (\omega_0) \to \mathrm{Div}^\ast (\omega)
\]
is a cochain isomorphism.
\end{lemma}

\begin{proof}
The lemma follows from the standard fact that
\[
e^{-S_\omega} \Delta_0 e^{S_\omega}= \Delta_\omega .
\]
Indeed, let $I \in \sO (T[-1](X, \fg))$, then we have
\[
\begin{array}{lll}
\Delta_0 (e^{S_\omega}  I) &=& \{ e^{S_\omega}, I\} - \Delta_0 (e^{S_\omega})  I +e^{S_\omega}  \Delta_0 (I) \\[2ex]
&=& e^{S_\omega} \{S_\omega , I\} - e^{S_\omega} \left (\Delta_0 (S_\omega) +1/2 \{S_\omega, S_\omega\} \right )  I + e^{S_\omega}  \Delta_0 (I) \\[2ex]
&=& e^{S_\omega} \{S_\omega, I\} + e^{S_\omega} \Delta_0 (I),
\end{array}
\]
where the last equality follows from $S_\omega$ satisfying the QME.

\end{proof}

We obtain as an immediate corollary of the lemma the following.

\begin{prop}\label{prop:int}
Let $(X, \fg)$ be an $\L8$ space and let $dVol_0$ be the canonical projective volume form on $T[-1] (X, \fg)$.  Further, let $dVol_{S^1}$ be the projective volume form determined by the one loop quantization of one dimensional Chern--Simons theory.  Then $dVol_{S^1}$ is integrable if and only if $dVol_0$ is integrable.
\end{prop}

So far we have not found a general criterion for the integrability of $dVol_0$ and hence $dVol_{S^1}$.  If our $\L8$ space is $(X , \fg_{X_{\overline{\partial}}})$ for a complex manifold $X$, then $dVol_0$ is integrable.  For $X$ smooth, the canonical volume form on $(X, \fg_X)$ is also integrable, but the integration map is subtle, as we discuss below.

\section{Integration}\label{sect:int}

Suppose that $(X, \fg)$ is a nice $\L8$ space and that $\omega$ is an integrable projective volume form on $T[-1] (X, \fg)$.  By definition the projection map of the shifted cotangent bundle
\[
\fg[\epsilon] \oplus (\fg[\epsilon])^\vee [-3] \to \fg[\epsilon]
\]
is a map of $\L8$ algebras and hence induces a natural pull back map on functions
\[
\sO(T[-1](X, \fg)) \to \sO (T^\ast[-1] T[-1] (X, \fg))
\]
which leads to a map of sheaves
\[
\sO (T[-1] (X, \fg)) \to \mathrm{Div}^\ast (\omega).
\]

We pass to compactly supported cohomology to obtain a map
\[
H^i_c (X, \sO(T[-1](X, \fg))) \to H^i_c (X, \mathrm{Div}^\ast (\omega))
\]
and taking $\CC^\times$ invariants gives
\[
H^i_c (X, \sO(T[-1] (X, \fg))) \to H^i_c (X, (\mathrm{Div}^\ast (\omega))^{\CC^\times}) \cong H^{i+2d_1}_c (X, \cD (\omega)).
\]
Now as $\cD(\omega)$ is isomorphic to the orientation local system on $X$, so the right hand side is zero unless $i +2d_1$ is the dimension of $X$ and in this dimension it is one-dimensional.

\begin{definition}
The integral associated to an integrable volume form $\omega$ on $T[-1](X, \fg)$, with $\dim X =n$, is the map
\[
H^{n-2d_1}_c (X, \sO (T[-1](X, \fg))) \to H^n_c (X, \cD(\omega)) \cong \CC.
\]
\end{definition}

\subsection{Integration on $T[-1](X, \fg)$}

We begin by recalling Costello's work with complex manifolds as the theorems have a particularly nice form in this setting.  To begin, Costello shows that for $(X, \fg_{X_{\overline{\partial}}})$ the $\L8$ space encoding the complex structure of $X$ (real dimension of $X$ is $2n$), then
\[
H^0  (X, \sO  (T[-1] (X, \fg_{X_{\overline{\partial}}})  )  ) = \oplus H^i (X, \Omega^n_{X, hol}).
\]

\begin{theorem}[Theorem 8.0.3 of \cite{CosWG}]
The projective volume form $dVol_0$ on $T[-1](X, \fg_{X_{\overline{\partial}}})$ is integrable and the associated integral
\[
\int : H^0  (X, \sO = (T[-1] (X, \fg_{X_{\overline{\partial}}})  )  ) \to \CC
\]
is (up to a scalar) the usual integration on $H^n(X, \Omega^n_{X, hol})$ and vanishes on $H^i(X, \Omega^{i}_{X, hol})$ for $i<n$.
\end{theorem}

From Theorem \ref{thm:QCS} we know that one dimensional Chern--Simons defines a projective volume form on $T[-1] (X, \fg_{X_{\overline{\partial}}})$ which we will denote $\overline{dVol}_{S^1}$.  That this volume form is integrable follows from Proposition \ref{prop:int}.  We then have the following.

\begin{theorem}
The quantization of one dimensional Chern--Simons theory with target a complex manifold $X$ determines an integrable projective volume form $\overline{dVol}_{S^1}$ on $T[-1](X, \fg_{X_{\overline{\partial}}})$ with associated integral
\[
\int_{T[-1](X, \fg_{X_{\overline{\partial}}})} \alpha \; \overline{dVol}_{S^1} = k \int_{X} \alpha \cdot e^{-c_1 (X)/2} Td(X) ,
\]
for $k$ a nonzero constant and where $\int_X$ indicates the usual integration of holomorphic forms.
\end{theorem}

For a general $\L8$ space we also obtain a volume form $dVol_{S^1}$ on $T[-1](X, \fg)$ from the quantization of one dimensional Chern--Simons theory (this is the one from the previous section). However, as we noted above we do not know a general integrability criterion.  In the case that $dVol_0$ and $dVol_{S^1}$ are both integrable we have the following.

\begin{theorem}\label{thm:thm35}
Let $(X, \fg)$ be an $\L8$ space such that $dVol_0$ is integrable, then 
\[
\int_{T[-1] (X, \fg)} \alpha \;  dVol_{S^1} = \int_{T[-1] (X, \fg)} \alpha \cdot \hat{A} (B \fg) dVol_0 .
\]
\end{theorem}

The two preceding theorems follow from a simple lemma.

\begin{lemma}
Let $\omega_0$ and $\omega$ be integrable volume forms on an $\L8$ space $(X, \fg)$ such that 
for the corresponding BV laplacians we have
\[
\Delta_\omega = \Delta_0 + \{ S_\omega , -\}.
\]
Then for $\alpha \in H^0 (X, \sO (X, \fg))$ we have
\[
\int_{(X, \fg)} \alpha \omega = \int_{(X, \fg)} \alpha e^{S_\omega} \omega_0 .
\]
\end{lemma}

\begin{proof}
We know from Lemma \ref{lem:gaugetransform} that multiplication by $e^{S_\omega/\hbar}$ is a cochain isomorphism between divergence complexes, hence we have a commutative diagram
\[
\xymatrix{ & H^i_c (X, \mathrm{Div}^\ast (\omega_0)^{\CC^\times}) \ar[2,0]_{\cong}^{e^{S_\omega/\hbar}} \\
H^i_c (X, \sO(X, \fg)) \ar[-1,1]^{\int \omega_0} \ar[1,1]_{\int \omega} \\
& H^i_c (X, \mathrm{Div}^\ast (\omega)^{\CC^\times} )}
\]

\end{proof}

\subsection{The Smooth Case}\label{sect:smooth}

In this final section we discuss integration on $T[-1] (X,\fg_X)$ for $X$ a smooth, compact, oriented manifold.  We will see that $dVol_0$ and hence $dVol_{S^1}$ are integrable, but the associated integrals are trivial.  We then take homotopy invariants with respect to the action of $B\GG_a$ in order to obtain a non trivial integration map. 

\begin{prop}
The canonical volume form $dVol_0$ on $T[-1](X, \fg_X)$ is integrable.
\end{prop}

\begin{proof}
We can identify the divergence complex $\mathrm{Div}^\ast (dVol_0)$ with $\Omega^{-\ast}_{T^\ast B \fg} ((\hbar))$ which is quasi-isomorpic to $dR(J( \Omega^{-\ast}_{T^\ast X}))((\hbar))$ which itself is quasi-isomorphic to $\Omega^{-\ast}_{T^\ast X}((\hbar))$.  Hence the cohomology sheaves of this complex are zero except in dimension $-2n$, where $\dim X =n$, in which they are $\CC((\hbar))$.  We see that $\cD(dVol_0)$ is the trivial local system and we assume $X$ is oriented and hence the proposition follows.
\end{proof}

Note that we know immediately that since $dVol_0$ is integrable, so is $dVol_{S^1}$.  

The integral associated to $dVol_0$ on $T[-1] (X, \fg_X)$ is a map
\[
\int : H^{-n} (\sO(T[-1](X, \fg_X))) \to H^n (X, \cD (dVol_0)) \cong \CC .
\]
But we have
\[
H^{-n} (\sO (T[-1](X, \fg_X))) =0.
\]
In order to have an interesting integral we must remember the action of $B\GG_a$, see Section 9.3 of \cite{GGCS}.  So far we have constructed a  degree 0 map of sheaves
\[
\int : \sO(T[-1] (X, \fg_X)) \to \CC[n] .
\]
We now pass to $B\GG_a$ (homotopy) invariants where we equip $\CC$ with the trivial $B\GG_a$ module structure.  Therefore, we obtain a map
\[
\int dVol_0^{B\GG_a} : (\sO(T[-1] (X, \fg_X)))^{B\GG_a} \simeq ( \Omega^{-\ast}_X[[u]], ud) \to \CC[[u]][n].
\]
As $dVol_{S^1}$ is also integrable, we obtain an equivariant integration map
\[
\int^{S^1} dVol_{S^1} : (\sO(T[-1] (X, \fg_X)))^{B\GG_a} \to \CC[[u]][n].
\]

\begin{theorem}
Let $(X, \fg_X)$ be the $L_\infty$ space determined by the (oriented and compact) smooth $n$-manifold $X$. Then,
\[
\int^{S^1} 1 \cdot dVol_{S^1} = u^n \hat{A} (X) \in \CC[[u]][n].
\]
\end{theorem}

\begin{proof}
Let $\left ( \mathrm{Div}^\ast (dVol_{S^1}) \right )^{S^1}$ be the equivariant divergence complex and $u$ a degree 2 parameter. Then we have a quasi-isomorphism
\[
\left ( \mathrm{Div}^\ast (dVol_{S^1} )\right )^{S^1} \simeq \left (\Omega^{-\ast}_{T^\ast X}[[u]]((\hbar)), u d_{T^\ast X} + \{ \log \hat{A}(X)_u , - \} \right ).
\]
(Recall the definition of $\log \hat{A}(X)_u$ from Appendix \ref{app:genera}.)
So $\int^{S^1} dVol_{S^1}$ is realized as the map at the level of cohomology induced by the map
\[
(\sO(T[-1] (X, \fg_X)))^{B\GG_a} \to \left ( \mathrm{Div}^\ast (dVol_{S^1} ) \right )^{S^1}
\]
and passing to $\CC^\times$-invariants.
The theorem then follows directly from Lemma \ref{lem:gaugetransform}.

\end{proof}

\appendix

\section{BV Theory a la Costello}\label{app:BV1}

In this appendix we recall a mathematical take on QFT in the Batalin--Vilkovisky formalism as described by Kevin Costello \cite{Cos1}.  Our presentation follows that of \cite{GLL} and \cite{SiVertex}.

\subsection{BV Algebras and Observables}\label{app:BV}

\begin{definition}\label{defn-BV} A BV algebra is a pair $(\cA, \Delta)$ where
\begin{itemize}
\item $\cA$ is a $\ZZ$-graded commutative associative unital algebra. 
\item $\Delta: \cA \to \cA$ is a second-order operator of degree $1$ such that $\Delta^2=0$. 
\end{itemize}
\end{definition}

Here $\Delta$ is called the BV operator. $\Delta$ being ``second-order" means the following: define the \emph{BV bracket} $\{-,-\}_\Delta$ as the measuring of the failure of $\Delta$ being a derivation
$$
    \{a,b\}_\Delta:=\Delta(ab)-(\Delta a)b- (-1)^{\lvert a \rvert}a \Delta b. 
$$
In this section we will suppress $\Delta$ from the notation, simply writing $\{-,-\}$. Then $\{-,-\}: \cA\otimes \cA\to \cA$ defines a Poisson bracket of degree $1$ satisfying
\begin{itemize}
\item $\{a,b\}=(-1)^{\lvert a \rvert \lvert b \rvert}\{b,a\}$. 
\item $\{a, bc\}=\{a,b\}c+(-1)^{(\lvert a \rvert+1)\lvert b \rvert}b\{a,c\}$. 
\item $\Delta\{a,b\}=-\{\Delta a, b\}-(-1)^{\lvert a \rvert}\{a, \Delta b\}$. 
\end{itemize}

\begin{definition} A differential BV algebra is a triple $(\cA, Q, \Delta)$ where
\begin{itemize}
\item $(\cA, \Delta)$ is a BV algebra (see Definition \ref{defn-BV}).
\item $Q: \cA\to \cA$ is a derivation of degree $1$ such that $Q^2=0$ and $[Q, \Delta]=0$. 
\end{itemize}

\end{definition}

\begin{definition}
Let $(\cA, Q, \Delta)$ be a differential BV algebra. A degree $0$ element $I_0\in \cA$ is said to satisfy the \emph{classical master equation} (CME) if 
$$
QI_0+\frac{1}{2} \{I_0,I_0\}=0.
$$ 
A degree $0$ element $I\in \cA[[\hbar]]$ is said to satisfy the \emph{quantum master equation} (QME) if 
$$
QI+\hbar \Delta I+\frac{1}{2}\{I,I\}=0.
$$
Here $\hbar$ is a formal (perturbative) parameter. 
\end{definition}

 The ``second-order" property of $\Delta$ implies that QME is equivalent to 
$$
   (Q+\hbar \Delta)e^{I/\hbar}=0. 
$$
If we decompose $I=\sum\limits_{g\geq 0}I_g\hbar^g$, then the $\hbar\to0$ limit of the QME recovers the CME: 
$
QI_0+\frac{1}{2}\{I_0, I_0\}=0.
$

A solution $I_0$ of the CME leads to a differential $Q+\{I_0,-\}$, which is usually called the BRST operator in physics. 

\begin{definition}
Let $(\cA, Q, \Delta)$ be a differential BV algebra and $I_0 \in \cA$ satisfy the CME.  Then the {\it complex of classical observables}, $Obs^{cl}$ is given by
\[
Obs^{cl} \overset{\text{def}}{=}( \cA, Q+\{I_0,-\}).
\]
\end{definition}

Similarly, a solution $I$ of the QME yields a differential and correspondingly a complex of quantum observables.

\begin{definition}
Let $(\cA, Q, \Delta)$ be a differential BV algebra and $I \in \cA[[\hbar]]$ satisfy the QME.  Then the {\it complex of quantum observables}, $Obs^{q}$ is given by
\[
Obs^{q} \overset{\text{def}}{=}( \cA[[\hbar]], Q+\hbar \Delta + \{I,-\}).
\]
\end{definition}

Note that $Obs^{cl}$ has a degree 1 Poisson bracket, so following \cite{CosGw2} we call it a $P_0$ algebra.  Similarly, in {\it ibid.} the structure on $Obs^q$ is called a BD algebra.

\subsection{Perturbative BV Quantization}\label{app:BV2}

The data of a classical field theory over a manifold $M$ consists of a graded vector bundle $E$ (possibly of infinite rank) equipped with a -1 symplectic pairing and a local functional $S \in \sO_{loc} (\cE)$\footnote{Note that $\sO (\cE) = \csym(\cE^\vee)$, the subspace $\sO_{loc} (\cE) \subset \sO(\cE)$ consists of those functionals determined by Lagrangian densities.} expressed as $S(e) = \langle e, Q(e) \rangle + I_0 (e)$, where $Q$ is a square zero differential operator of cohomological degree 1, such that
\begin{enumerate}
\item $S$ satisfies the CME, i.e., $\{S, S\} = 0$;
\item $I_0$ is at least cubic; and
\item $(\cE, Q)$ is an elliptic complex.
\end{enumerate}

\begin{definition}
A classical field theory $(\cE, S)$ over $M$ is a {\it cotangent theory} if we can write the field content as
\[
\cE = \Gamma \left (M ; E[1] \oplus \left (E^\vee \otimes \mathrm{Dens}(M) [-2] \right ) \right ).
\]
We further require that the action $S$ vanishes on tensors where there are at least two sections from the second summand $E^\vee \otimes \mathrm{Dens} (M)$.
\end{definition}

Quantization of a field theory $(\cE, S)$ over $M$ consists of two stages:
\begin{enumerate}
\item Build a BV algebra from the data of the pairing on the bundle $E$; and
\item Promote the classical action $S$ to a solution of the QME in this BV algebra.
\end{enumerate}

The first difficulty is that the Poisson kernel $K$ dual to the symplectic pairing is nearly always singular, so the naive definition of the BV operator $\Delta_k = \partial_k$ is ill-defined. In \cite{Cos1}, Costello uses homotopical ideas (built on the heat kernel) to build a family of well defined (smooth) BV operators $\Delta_L$ for $0<L<\infty$. Consequently, there is a family of differential BV algebras $\left \{ (\sO(E), Q , \Delta_L)\right \}_{L>0}$. Costello also describes {\it homotopy renormalization group flow} (HRG) to relate solutions of the QME between algebras in this family. As we describe in the next section,  HRG is expressed in terms of a propagator built from the differential operator $Q$ and a {\it gauge fixing operator} $Q^\dagger$; indeed, any parametrix for the generalized Laplacian $[Q, Q^\dagger]$ can be used as a propagator. 

\begin{definition}
Let $(\cE, S)$ be a classical field theory over $M$.  A {\it perturbative quantization} is a family of solutions to the QME, $\{I[L]\}_{L>0}$, linked by the HRG, such that 
\[
\lim_{L \to 0} I[L] \equiv I_0 \quad \quad (\text{modulo } \hbar).
\]
\end{definition}

\begin{remark}
Flow via the HRG induces a chain homotopy between quantum observables  as we vary within the family of BV algebras $\{(\sO(E), Q, \Delta_L)\}_{L>0}$.  Thus, we will supress the dependence on $L$ and abusively refer to these chain homotopic complexes as the {\it global quantum observables} of our field theory.
\end{remark}

\subsection{Homotopy Renormalization Group Flow}\label{app:hrg}

The homotopy renormalization group flow equation can be described in terms of Feynman graphs. Note that our description is for an arbitrary functional on a space of fields $\cE$. Further, we will work relative to an arbitrary dg algebra $\cA$ equipped with a nilpotent ideal $\cI$.

\begin{definition}
A graph $\cG$ consists of the following data:
\begin{enumerate}
 \item A finite set of vertices $V(\cG)$;
 \item A finite set of half-edges $H(\cG)$;
 \item An involution $\sigma: H(\cG)\rightarrow H(\cG)$. The set of fixed points of this map is denoted by $T(\cG)$ and is
called the set of tails of $\cG$. The set of two-element orbits is denoted by $E(\cG)$ and is called the set of internal edges of
$\cG$;
 \item A map $\pi:H(\cG)\rightarrow V(\cG)$ sending a half-edge to the vertex to which it is attached;
 \item A map $g:V(\cG)\rightarrow \mathbb{Z}_{\geqslant 0}$ assigning a genus to each vertex.
\end{enumerate}
\end{definition}
It is clear how to construct a topological space $|\cG|$ from the above abstract data. A graph $\cG$ is called $connected$ if
$|\cG|$ is connected. The genus of the graph $\cG$ is defined to be 
\[
g(\cG):=b_1(|\cG|)+\sum_{v\in V(\cG)}g(v),
\] 
where
$b_1(|\cG|)$ denotes the first Betti number of $|\cG|$. Let 
\[
\sO^+(\cE)\subset \sO(\cE)[[\hbar]]
\] 
be the subspace consisting of those
functionals  which are at least cubic modulo $\hbar$ and the nilpotent ideal $\mathcal{I}$ in the base ring
$\mathcal{A}$.  Let $F\in \sO^+(\cE)$ be a functional,  which can be expanded as
\[
F=\sum_{g,k\geq 0}\hbar^g F_{g}^{(k)}, \quad F_{g}^{(k)}\in \sO^{(k)}(\cE).
\]
We view each $F_{g}^{(k)}$ as an
$S_k$-invariant linear map
\[
F_{g}^{(k)}: \mathcal{E}^{\otimes k}\rightarrow\mathcal{A}.
\]
With the propagator $P_{\epsilon \to L}$, we will describe the {\it (Feynman) graph weights}
\[
W_\cG(P_{\epsilon \to L},F)\in \sO^+(\cE)
\] 
for any connected graph $\cG$. We label each vertex $v$ in $\cG$ of genus $g(v)$ and valency $k$ by
$F^{(k)}_{g(v)}$. This defines an assignment
\[
F(v):\mathcal{E}^{\otimes H(v)}\rightarrow \cA,
\]
where $H(v)$ is the set of half-edges of $\cG$ which are incident to $v$.
Next, we label each internal edge $e$ by the propagator 
\[
P_e=P_{\epsilon \to L}\in\mathcal{E}^{\otimes H(e)},
\]
where $H(e)\subset H(\cG)$ is the two-element set consisting of the half-edges forming $e$. We can then contract
\[
\otimes_{v\in V(\cG)}F(v): \mathcal{E}^{H(\cG)}\rightarrow \cA
\]
with 
\[
\otimes_{e\in E(\cG)} P_e\in\mathcal{E}^{H(\cG)\setminus T(\cG)}
\] 
to yield a linear map
\[
W_\cG(P_{\epsilon \to L},F) : \mathcal{E}^{\otimes T(\cG)}\rightarrow \cA.
\]

\begin{definition}
We define the (homotopy) RG flow operator with respect to the propagator $P_{\epsilon \to L}$ 
\[
   W(P_{\epsilon \to L}, -): \sO^+(\cE)\to \sO^+(\cE), 
\]
by
\begin{equation}\label{RG-flow}
W(P_{\epsilon \to L}, F):=\sum_{\cG}\frac{\hbar^{g(\cG)}}{\lvert \text{Aut}(\cG)\rvert}W_\cG(P_{\epsilon \to L}, F)
\end{equation}
where the sum is over all connected graphs.
\end{definition}

Equivalently, it is useful to describe the (homotopy) RG flow operator formally via the simple equation 
\[
e^{W(P_{\epsilon \to L}, F)/\hbar}=e^{\hbar \partial_{P_{\epsilon \to L}}} e^{F/\hbar}.
\]

\begin{definition} A family of functionals $F[L] \in \sO^+(\cE)$ parametrized by $L>0$ is said to satisfy the homotopy renormalization group flow equation (hRGE) if for each $0 < \epsilon < L$
\[
    F[L]=W(P_{\epsilon \to L}, F[\epsilon]).
\]
\end{definition}

\subsection{BV Theory and Volume Forms: Integration via Homology}\label{sect:pvol}

We now quickly explain the underlying relationship between BV theory, path integrals, homological methods, and perturbation theory. The story below suggests a relationship between quantizations of a cotangent theory and projective volume forms; the precise relationship is given by Proposition \ref{prop:projvol} above. More detailed (and eloquent) presentations are given in \cite{ABF}, \cite{Fiorenza}, and \cite{StasheffBV}. Also, historically much of the mathematical development goes back to  Koszul \cite{Koszul}.

For simplicity, let $X$ be a connected, orientable, smooth manifold of dimension $n$.\footnote{Using densities, the following arguments can be adopted to unoriented manifolds.} Every top form $\mu \in \Omega^n(X)$ then defines a linear functional
\[
\begin{array}{cccc}
\int_\mu:& \cinf_c(X)& \to &\RR \\
 & f & \mapsto & \int_X f \mu
 \end{array},
\]
which is a natural object from several perspectives. First, from this linear functional --- the distribution associated to $\mu$ --- we can completely reconstruct the top form $\mu$. Second, if $\mu$ is a probability measure, then $\int_\mu$ is precisely the expected value map. Our goal is now to rephrase $\int_\mu$ in a way that does not explicitly depend on ordinary integration and thus to obtain a version of volume form that can be extended to $\L8$ spaces.

We can understand $\int_\mu$ in a purely homological way, as follows. We know that integration over $X$ vanishes on total derivatives $d\omega \in \Omega^n_c(X)$, by Stokes' Theorem, so we have a commutative diagram
\[
\xymatrix{
\Omega^n_c(X) \ar[rr]^{\int_X} \ar[rd]_{[-]} & & \RR \\
 & H^n_c(X) \ar[ru]_{\cong} &
}
\]
where $[\omega]$ denotes the cohomology class of the top form $\omega$. (The cohomology group $H^n_c(X)$ is 1-dimensional by Poincar\'e duality.) In consequence, we can identify $\int_\mu$ with the composition
\[
\xymatrix{
 \Omega^n_c(X) \ar[r]^{[-]} &  H^n_c(X)\\
\cinf_c(X) \ar[u]^{\iota_\mu} \ar[ur]_{\int_\mu}&
}
\]
where $\iota_\mu$ denotes ``multiplication by $\mu$" (or ``contraction with $\mu$"). We thus have a purely homological version of integration against $\mu$.

It is natural to extend the map ``contract with $\mu$" to the whole de Rham complex, and not just the top forms:
\[
\xymatrix{ 
\dotsb \ar[r] & \Omega_c^{n-2}(X) \ar[r]^d & \Omega_c^{n-1}(X) \ar[r]^d & \Omega^n_c (X)\ar[r]^{\int_X} & \RR \\
\dotsb \ar[r] & PV^2_c T(X)  \ar[r]^{div_\mu} \ar[u]_{\iota_{\mu}} & PV_c^1(X)  \ar[r]^{div_\mu} \ar[u]_{\iota_{\mu}} & C_c^\infty(X) \ar[u]_{\iota_{\mu}}  \ar[ur]_{\int_\mu}
},
\]
where $PV^k_c(X) := \Gamma_c(X, \Lambda^k T_X)$ denotes the compactly-supported {\em polyvector fields} and $div_\mu$ denotes ``divergence with respect to $\mu$." We require now that $\mu$ is nowhere-vanishing, so that the divergence is well-defined. This map of cochain complexes $\iota_\mu$ is then an isomorphism.

The significance of the bottom row is that it fully encodes integration against $\mu$ but the relevant data of $\mu$ is contained in the differential $div_\mu$. We would like to characterize such differentials on the polyvector fields $PV(X)$, in order to describe a version of integration that applies to spaces more general than manifolds (at least spaces that possess a good notion of polyvector fields).\footnote{This kind of integration notion would work even for infinite-dimensional spaces, for which there are no top forms but there is a ring of functions.}

Note that a choice of top form $\mu$ induces a $\cinf_X$-linear map from $\cinf_X$ to $\Omega^n_X$, and so we can pullback the natural {\em right} $D_X$-module structure on $\Omega^n_X$.\footnote{Differential operators $D_X$ naturally act on $\cinf_X$ from the left, and they naturally act on top forms from the right. This right action is a consequence of the integration pairing between top forms and functions. In other words, $D_X$ acts on distributions from the right, and hence on top forms as well.} Let $\cinf_{X,\mu}$ denote $\cinf_X$ equipped with this right $D_X$-module structure.\footnote{Let $p: X \to \pt$ denote the map to a point. Then the derived pushforward of $\cinf_{X,\mu}$ along $p$ is given by $(PV(X), div_\mu)$. This construction is sometimes referred to as the {\it Spencer resolution}.}
Although a volume form makes $\cinf_X$ into a right $D_X$-module, the converse need not hold. Observe that for any nonzero constant $c$, the operators $div_\mu$ and $div_{c \mu}$ are the same. Hence, a right $D_X$-module structure on $\cinf_X$ locally gives a volume form only {\em up to scale}. 

The complex $(PV(X), div_\mu)$ is a fundamental example of a BV algebra (see Appendix \ref{app:BV}), the associated bracket $\{-,-\}$ is the Schouten bracket.  Further, note that we have an equivalence
\[
 \sO (T^\ast [-1]X) \cong PV(X).
\]
Hence, we witness the intimate relationship between BV algebras, shifted cotangent bundles, integration via homological algebra, and (projective) volume forms.

\section{The $\hat{A}$ Genus}\label{app:genera}

Recall that (following Hirzebruch \cite{H}) the Todd class can be defined in terms of Chern classes by the power series $Q (x)$ and the $\hat{A}$ class is given in Pontryagin classes via $P (x)$ where
\[
Q(x) = \frac{x}{1- e^{-x}} \text{   \; \; \;    and     \; \; \;      } P(x) = \frac{x/2}{\sinh x/2} .
\]
We define a new power series by $\log (Q(x)) - x/2$ and denote the corresponding characteristic class by $\log (e^{-c_1/2} \text{Td})$. We have an equivalence of power series (see \cite{Cartier} and \cite{BHJ})
\begin{equation}\label{pwrsrs}
\log \left (\frac{x}{1-e^{-x}} \right ) - \frac{x}{2} = \sum_{k \ge 1} 2 \zeta (2k) \frac{x^{2k}}{2k (2 \pi i)^{2k}} ,
\end{equation}
where $\zeta$ is the Riemann zeta function.

We now use standard arguments about characteristic classes. For a sum of complex line bundles $E = L_1 \oplus \cdots \oplus L_n$, the Todd class is 
\[
Td(E) = Q(c_1(L_1)) \cdots Q(c_1(L_n)).
\]
Thus, equation \ref{pwrsrs} tells us
\[
\log (e^{-c_1(E)/2} Td(E) ) = \sum_{k \ge 1} \frac{2 \zeta (2k)}{2k (2 \pi i)^{2k}} (c_1(L_1)^{2k} + \cdots + c_1(L_n)^{2k}).
\] 
As $ch_{2k}(E) = (c_1(L_1)^{2k} + \cdots + c_1(L_n)^{2k})/(2k!)$, we obtain a general formula for an arbitrary bundle $E$,
\[
\log (e^{-c_1(E)/2} Td(E) ) = \sum_{k \ge 1} \frac{2 \zeta (2k)}{2k (2 \pi i)^{2k}} (2k)! ch_{2k}(E).
\]

Putting together the above discussion we make the following definition for an $L_\infty$ space $B \fg$.

\begin{definition}
Let $V$ be a vector bundle over $(X, \fg)$ (e.g. the tangent bundle as given by the module $\fg[1]$) then the we define
\[
\log(\hat{A} (V)) \overset{\text{def}}{=} \sum_{k \ge 1} \frac{2 \zeta (2k)}{2k (2 \pi i)^{2k}} (2k)! ch_{2k}(V) \in \Omega^{-\ast}_{B \fg} .
\]
\end{definition}

Similarly, we will need an equivariant class, so we define for any smooth manifold $X$ the class $\log (\hat{A}_u (X))$ to be 
\[
\log (\hat{A}_u (X)) \overset{def}{=} \sum_{k \ge 1} \frac{2 (2k-1)!}{(2 \pi i)^{2k}} u^{2k} \zeta(2k) ch_{2k} (X) \in (\Omega^{-\ast}_X [[u]], ud ).
\]
This is the usual logarithm of the $\hat{A}$ class  weighted by powers of $u$.

\subsection{The Many Faces of the $\hat{A}$ Genus}

The $\hat{A}$ genus/class makes many appearances in math and physics, as we now briefly sketch.

\subsubsection{Index of the Dirac Operator}

In the late 1950s Borel and Hirzebruch \cite{BH} proved that $\hat{A} (X)$ was an integer provided that $M$ was a spin manifold.  The question of why the spin condition implied integrality was quickly sorted by Hirzebruch and Atiyah \cite{AH}.
In extended work with Singer (and Bott and Hirzebruch) \cite{AS}, Atiyah fleshed out this picture through the formulation (and proof) of the {\it index theorem}. In particular, for a spin manifold $X$, the integer $\hat{A} (X)$ is the Fredholm index of an elliptic operator: the Dirac operator.

\subsubsection{An Obstruction to Positive Scalar Curvature}

Using a Bochner/Weitzenb\"{o}ck formula, Lichnerowicz \cite{Lich} showed that $\hat{A} (X)$ is an obstruction to a spin manifold $X$ admitting a metric of positive scalar curvature. Hitchin \cite{Hitchin} refined this result and introduced a K-theoretic invariant, the $\alpha$ invariant. Significant progress on the converse statement was made by Gromov and Lawson \cite{GL}, before the problem was solved--in the simply connected case--by Stolz \cite{Stolz}.

\subsubsection{Twisted Todd Class}
In the holomorphic setting we use the characteristic class $e^{-c_1(E)/2} Td(E)$. Note that for real bundles, this agrees with $\hat{A} (E)$. Following \cite{BHJ} or \cite{Gilkey} there is an index theoretic explanation. Indeed, suppose that a manifold $X$ is both complex and spin.  Then there are two canonical $\text{spin}^c$ structures and corresponding Dirac operators.  Let $\dirac^\CC$ be the Dirac operator corresponding to the $\text{spin}^c$ structure determined by the complex structure and let $\dirac$ denote the Dirac operator determined by the spin structure.  One can show that $\dirac^\CC = \dirac_{K^{1/2}}$ where $K^{1/2}$ denotes the square root of the canonical bundle.  Then via the Atiyah-Singer index theorem we have
\[
\langle \mathrm{Td} (T_X) , [X] \rangle = \text{Index} \; \dirac^\CC = \text{Index} \; \dirac_{K^{1/2}} = \langle e^{c_1 (TX) /2} \hat{A} (T_X) , [X] \rangle .
\]

\subsubsection{As a Partition Function}

The relationship between the $\hat{A}$ class/genus and (supersymmetric) physics has an extended history, for highlights see \cite{AG}, \cite{FW}, or \cite{GetzlerIndex}.  Indeed, in our previous work, \cite{GGCS}, we demonstrated $\hat{A}(M)$ as the partition function of a one dimensional perturbative BV theory.

%
%
%
%
%

\bibliographystyle{amsalpha}
\bibliography{volume_form}

\end{document}